\documentclass[uplatex,a4paper,12pt]{article}
\usepackage{amssymb}
\usepackage{amsmath}
\usepackage{amsthm}
\usepackage[pdftex]{graphicx}
\usepackage{tikz}
\usepackage{mathtools}
\usepackage[backend=biber,style=alphabetic,maxcitenames=8,mincitenames=4,maxbibnames=1000,minbibnames=10,sorting=nty,doi=false,isbn=false,url=false,eprint=true]{biblatex}
\usepackage{geometry}
\usepackage{here}
\usepackage[affil-it]{authblk}

\renewbibmacro{in:}{}
\DeclareMathAlphabet{\mymathbb}{U}{BOONDOX-ds}{m}{n}

\newcommand{\range}{\operatorname{ran}}
\newcommand{\dom}{\operatorname{dom}}

\newcommand{\frakc}{\mathfrak{c}}
\newcommand{\frakb}{\mathfrak{b}}
\newcommand{\frakd}{\mathfrak{d}}

\newcommand{\Pow}{\mathcal{P}}
\newcommand{\non}{\operatorname{non}}
\newcommand{\cov}{\operatorname{cov}}
\newcommand{\add}{\operatorname{add}}

\newcommand{\nul}{\mathcal{N}}
\newcommand{\meager}{\mathcal{M}}

\newcommand{\scrA}{\mathcal{A}}

\newcommand{\ZFC}{\mathrm{ZFC}}

\newcommand{\frakr}{\mathfrak{r}}

\newcommand{\bI}{\mathfrak{b}_\mathrm{game}^{\mathrm{I}}}
\newcommand{\bII}{\mathfrak{b}_\mathrm{game}^{\mathrm{II}}}
\newcommand{\dI}{\mathfrak{d}_\mathrm{game}^{\mathrm{I}}}
\newcommand{\dII}{\mathfrak{d}_\mathrm{game}^{\mathrm{II}}}
\newcommand{\rI}{\mathfrak{r}_\mathrm{game}^{\mathrm{I}}}
\newcommand{\rII}{\mathfrak{r}_\mathrm{game}^{\mathrm{II}}}
\newcommand{\aNI}{\add(\nul)_\mathrm{game}^{\mathrm{I}}}
\newcommand{\aNII}{\add(\nul)_\mathrm{game}^{\mathrm{II}}}
\newcommand{\bIstar}{\mathfrak{b}_\mathrm{game^\ast}^{\mathrm{I}}}
\newcommand{\bIIstar}{\mathfrak{b}_\mathrm{game^\ast}^{\mathrm{II}}}
\newcommand{\dIstar}{\mathfrak{d}_\mathrm{game^\ast}^{\mathrm{I}}}
\newcommand{\dIIstar}{\mathfrak{d}_\mathrm{game^\ast}^{\mathrm{II}}}

\newcommand{\rIstar}{\mathfrak{r}_\mathrm{game^\ast}^{\mathrm{I}}}
\newcommand{\rIIstar}{\mathfrak{r}_\mathrm{game^\ast}^{\mathrm{II}}}
\newcommand{\aNIstar}{\add(\nul)_\mathrm{game^\ast}^{\mathrm{I}}}
\newcommand{\aNIIstar}{\add(\nul)_\mathrm{game^\ast}^{\mathrm{II}}}

\newcommand{\append}{{}^\frown}
\newcommand{\suc}{\operatorname{succ}}
\newcommand{\idealI}{\mathcal{I}}
\newcommand{\idealJ}{\mathcal{J}}

\newcommand{\zero}{\mymathbb{0}}
\newcommand{\one}{\mymathbb{1}}

\newcommand{\seq}[1]{{\langle#1\rangle}}

\DeclarePairedDelimiter\abs{\lvert}{\rvert}
\DeclarePairedDelimiterX{\norm}[1]{\lVert}{\rVert}{#1}

\theoremstyle{definition}
\newtheorem{thm}{Theorem}[section]
\newtheorem*{thm*}{Theorem}
\newtheorem{defi}[thm]{Definition}
\newtheorem*{defi*}{Definition}
\newtheorem{lem}[thm]{Lemma}
\newtheorem*{lem*}{Lemma}
\newtheorem{fact}[thm]{Fact}
\newtheorem*{fact*}{Fact}
\newtheorem{prop}[thm]{Proposition}

\newtheorem*{rmk*}{Remark}
\newtheorem{cor}[thm]{Corollary}
\newtheorem*{cor*}{Corollary}
\newtheorem*{convention*}{Convention}
\newtheorem{question}{Question}

\title{Game-theoretic variants of cardinal invariants}

\author{Jorge Antonio Cruz Chapital}
\affil{Departament of Mathematics, Univeristy of Toronto, 27 King's College Cir, Toronto, Canada. E-mail: chapi@matmor.unam.mx}

\author{Tatsuya Goto\thanks{Supported by JSPS KAKENHI Grant Number JP22J20021}}
\affil{Graduate School of System Informatics, Kobe University, 1-1 Rokkodai, Nada-ku, 657-8501 Kobe, Japan. E-mail: 202x603x@stu.kobe-u.ac.jp}

\author{Yusuke Hayashi\thanks{Supported by JST SPRING, Japan Grant Number JPMJSP2148}}
\affil{Graduate School of System Informatics, Kobe University, 1-1 Rokkodai, Nada-ku, 657-8501 Kobe, Japan. E-mail: 219x504x@stu.kobe-u.ac.jp}

\date{\today}

\begin{document}
	\maketitle

    \begin{abstract}
        We investigate game-theoretic variants of cardinal invariants of the continuum. The invariants we treat are the reaping number $\mathfrak{r}$, the bounding number $\mathfrak{b}$, the dominating number $\mathfrak{d}$, and the additivity number of the null ideal $\operatorname{add}(\mathsf{null})$.
        We also consider games, called tallness games, defined according to ideals on $\omega$ and characterize that each of Player I and Player II has a winning strategy.
    \end{abstract}

    \section{Introduction}

    The study of cardinal invariants of the continuum is important in set theory of reals. On the other hand, the study of infinite games is also an important topic in set theory.
    We study variants of cardinal invariants using infinite games.
    The invariants we treat are the reaping number $\frakr$, the bounding number $\frakb$, the dominating number $\frakd$, and the additivity number of the null ideal $\add(\nul)$.
    Furthermore, in a forthcoming paper, \cite{cghy}, the authors and T. Yamazoe will consider game-theoretic variants of splitting numbers along the lines of this paper.
    
    Depending on the definition of each cardinal invariant, there are normal versions of games and *-versions of games, and we consider 8 games in total.

    In the normal version, Player II must in each turn say $0$ or $1$. Player II wins if there is a real in the prescribed family and the values of this real at the points where Player II played $1$ have the given  relation to the natural number that Player I played.
    In contrast, in the *-version, Player II must in each turn play a natural number. Player II wins if the real consisting of the play of Player II is in the prescribed family and this real has the given relation to the real consisting of Player I's moves.

    For each game, two cardinal invariants are defined: the minimum size of a family such that Player II has a winning strategy and the minimum size of a family such that Player I has no winning strategy.

    Figure \ref{table:ourresults} summarizes our results.

    {
    \renewcommand{\arraystretch}{1.2}
    \begin{figure}
        \begin{table}[H]
            \begin{center}
            \begin{tabular}{c|cc}
            game & $\mathfrak{x}_\mathrm{game}^{\mathrm{I}}$ & $\mathfrak{x}_\mathrm{game}^{\mathrm{II}}$ \\ \hline
            bounding & $\frakb$ & $\frakd$ \\
            bounding* & $\frakb$ & $\frakc$ \\
            dominating & $\frakd$ & $\frakd$ \\
            dominating* & $\frakd$ & $\frakc$ \\
            reaping & $\max{ \{ \frakr , \frakd \} } \leq {?} \leq \max{ \{ \frakr_{\sigma}, \frakd \} }$ & $\frakc$ \\
            reaping* & $\infty$ & $\infty$ \\
            anti-localizing & $\add(\nul)$ & $\cov(\meager)$ \\
            anti-localizing* & $\add(\nul)$ & $\frakc$
            \end{tabular}
            \end{center}
        \end{table}
        \caption{Our results}\label{table:ourresults}
    \end{figure}
    }

    In addition to investigating cardinal invariants, in Section \ref{sec:tallness}, we study games defined according to ideal on $\omega$ regarding its tallness. Moreover, in Section \ref{sec:twoversions}, we generalize the results in Section \ref{sec:tallness} to show the Definable Ideal Dichotomy.

    Game-theoretic considerations of cardinal invariants can be found in \cite{kada2000more}, \cite{brendle2019construction}, and \cite{meagersetsinfinite} but our approach differs from these.

    In the rest of this section, we fix our notation.

    $(\forall^\infty n)$ and $(\exists^\infty n)$ are abbreviations to say ``for all but finitely many $n$" and ``there exist infinitely many $n$", respectively.

    For $A, B \subseteq \omega$, the relation $A \subseteq^* B$ means that $A \setminus B$ is finite. We say $B$ almost contains $A$ if $A \subseteq^* B$ holds.
    In addition, for $x, y \in \omega^\omega$, the relation $x \le^* y$ means $(\forall^\infty n)(x(n) \le y(n))$. We say $y$ dominates $x$ if $x \le^* y$ holds.
    
    $\zero$ is the set of all eventually zero sequences and $\one$ is that of eventually one sequences.

    $\frakc$ denotes the cardinality of the continuum.
    
    The following is the standard definition of cardinal invariants.

    \begin{defi}
        \begin{enumerate}
            \item $\scrA \subseteq \omega^\omega$ is a dominating family if for every $x \in \omega^\omega$, there is $y \in \scrA$ that dominates $x$. Define the dominating number $\frakd$ by $\frakd = \min \{ \abs{\scrA} : \scrA \subseteq \omega^\omega \text{ a dominating family} \}$.
            \item $\scrA \subseteq \omega^\omega$ is an unbounded family if for every $x \in \omega^\omega$, there is $y \in \scrA$ that is not dominated by $x$. Define the bounding number $\frakb$ by $\frakb = \min \{ \abs{\scrA} : \scrA \subseteq \omega^\omega \text{ an unbounded family} \}$.
            \item For $x \in \Pow(\omega)$ and $y \in [\omega]^\omega$, we say $y$ reaps $x$ if either $y \subseteq^* x$ or $y \subseteq^* \omega \setminus x$ holds. it is equivalent to say $x$ does not split $y$. $\scrA \subseteq [\omega]^\omega$ is a reaping family if for every $x \in \Pow(\omega)$, there is $y \in \scrA$ such that $y$ reaps $x$. Define the reaping number $\frakr$ by $\frakr = \min \{ \abs{\scrA} : \scrA \subseteq [\omega]^\omega \text{ a reaping family} \}$.
            \item $\scrA \subseteq [\omega]^\omega$ is a $\sigma$-reaping family if for every $f \in (\Pow(\omega)^\omega$, there is $y \in \scrA$ such that $y$ reaps $f(n)$ for every $n \in \omega$. Define the $\sigma$-reaping number $\frakr_\sigma$ by $\frakr_\sigma = \min \{ \abs{\scrA} : \scrA \subseteq [\omega]^\omega \text{ a $\sigma$-reaping family} \}$.
            \item $\add(\nul)$ is the minimum cardinality $\kappa$ such that the Lebesgue null ideal is not $\kappa$-additive.
            \item $\cov(\meager)$ is the minimum cardinality $\kappa$ such that the Cantor space $2^\omega$ can be covered by $\kappa$ many meager sets.
        \end{enumerate}
    \end{defi}

    As for the details of these cardinal invariants, see \cite{blass2010combinatorial}.
	
	\section{Bounding games}

    In this section, we consider games related to unbounded families.
 
	Fix a set $\scrA \subseteq \omega^\omega$.
	We call the following game the \textit{bounding game} with respect to $\scrA$:
	
	\begin{table}[H]
		\centering
		\begin{tabular}{l|llllll}
			Player I  & $n_0$ &       & $n_1$ &       & $\dots$ &     \\  \hline
			Player II &       & $i_0$ &       & $i_1$ &         & $\dots$
		\end{tabular}
	\end{table}
	
	Here, $\seq{n_k : k  \in \omega}$ is a sequence of numbers in $\omega$ and $\seq{i_k : k  \in \omega}$ is a sequence of numbers in $2$.
	Player II wins when Player II played $1$ infinitely often and there is $g \in \scrA$ such that 
	\begin{align*}
		\{ k \in \omega :  i_k = 1 \} = \{ k \in \omega : n_k < g(k) \}.
	\end{align*}

	We call the following game the \textit{bounding* game} with respect to $\scrA$:
	
	\begin{table}[H]
		\centering
		\begin{tabular}{l|llllll}
			Player I  & $n_0$ &       & $n_1$ &       & $\dots$ &     \\  \hline
			Player II &       & $m_0$ &       & $m_1$ &         & $\dots$
		\end{tabular}
	\end{table}
 
	Here, $\seq{n_k : k  \in \omega}$ and $\seq{m_k : k  \in \omega}$ are sequences of numbers in $\omega$.
    Player II wins when
	\begin{align*}
        \seq{m_k : k  \in \omega} \in \mathcal{A} \text{ and } (\exists^\infty k)(n_k < m_k).
    \end{align*}

    \begin{defi}\label{defi:bIetc}
    	We define
    	\begin{align*}
    		\bI &= \min \{ \abs{\scrA} : \text{Player I has no winning strategy} \\
            & \hspace{15ex} \text{for the bounding game with respect to } \scrA \}, \\
    		\bII &= \min \{ \abs{\scrA} : \text{Player II has a winning strategy} \\
            & \hspace{15ex} \text{for the bounding game with respect to } \scrA \}, \\
    		\bIstar &= \min \{ \abs{\scrA} : \text{Player I has no winning strategy} \\
            & \hspace{15ex} \text{for the bounding* game with respect to } \scrA \}, \text{ and} \\
    		\bIIstar &= \min \{ \abs{\scrA} : \text{Player II has a winning strategy} \\
            & \hspace{15ex} \text{for the bounding* game with respect to } \scrA \}.
    	\end{align*}
     \end{defi}

    Since the star version is harder for Player II than the non-star version, we have the following inequality.
    \[
    \tikz{
    \node (bI) at (0, 0) {$\bI$};
    \node (bIstar) at (3, 0) {$\bIstar$};
    \node (bII) at (0, 1.5) {$\bII$};
    \node (bIIstar) at (3, 1.5) {$\bIIstar$};
    \node at (1.5, 0) {$\le$};
    \node at (1.5, 1.5) {$\le$};
    \node at (0, 0.75) {\rotatebox{90}{$\le$}};
    \node at (3, 0.75) {\rotatebox{90}{$\le$}};
    }
    \]
 
	\begin{thm}
		$\bI = \frakb$ holds.
	\end{thm}
	\begin{proof}
		That $\bI \ge \frakb$ is easy. 
		We show $\bI \le \frakb$.
		Take an unbounded family $\mathcal{A} \subseteq \omega^\omega$.
		Take Player I's strategy $\sigma \colon 2^{<\omega} \to \omega$.
		We want to show that $\sigma$ is not a winning strategy for the bounding game with respect to $\mathcal{A}$.
		
		Since $\zero$ is a countable set, we can get $f \in \omega^\omega$ that dominates $\seq{\sigma(\bar{i} \upharpoonright k) : k \in \omega}$ for every $\bar{i} \in \zero$.
		Since $\mathcal{A}$ is an unbounded family, we can take $g \in \mathcal{A}$ such that $f$ doesn't dominate $g$.
		We now put $\bar{i} \in 2^\omega$ by
		\[
		i_k = \begin{cases} 1 & \text{(if $\sigma(\bar{i} \upharpoonright k) < g(k))$} \\ 0 & \text{(otherwise)} \end{cases}
		\]
		If $\bar{i} \in \zero$, then $\seq{\sigma(\bar{i} \upharpoonright k) : k \in \omega}$ does not dominate $g$ by the choice of $g$.
		But this fact and the choice of $\bar{i}$ imply $\bar{i} \not \in \zero$. It's a contradiction. So $\bar{i} \not \in \zero$. Therefore $\bar{i}$ is a play of Player II that wins against Player I's strategy $\sigma$.
	\end{proof}
	
	\begin{thm}\label{thm:bII}
		$\bII = \frakd$ holds.
	\end{thm}
	\begin{proof}
		We first prove $\bII \le \frakd$.
		Take a dominating family $\scrA \subseteq \omega^\omega$ of $(\omega^\omega, \le)$ (the total domination order).
		Then the strategy that plays $1$ always is a winning strategy for Player II.
		
		We next prove $\frakd \le \bII$.
		Fix $\scrA \subseteq \omega^\omega$ with a winning strategy of Player II for the bounding game with respect to $\scrA$.
		Consider the game tree $T$ decided by the winning strategy. So every node in $T$ of even length has full successor nodes and every node in $T$ of odd length has the only successor node determined by the strategy.
		We first consider the next case:
		\begin{itemize}
			\item (Case 1) There is a $\sigma \in T$ of even length such that for every even number $r \ge \abs{\sigma}$, there is $i \in 2$ such that for all but finitely many $m$, for every $\tau \in T$ extending $\sigma$, we have $[\tau(r) = m \implies \tau(r+1) = i]$.
		\end{itemize}
		Fix a witness $\sigma$ and $\seq{i_r : r \ge \abs{\sigma} \text{ even}}$ for Case 1.
		
		Then we have $(\exists^\infty r)(i_r = 1)$.
		Otherwise, we have $(\forall^\infty r)(i_r = 0)$.
		Then considering an appropriate play of Player I, Player II plays $0$ eventually along the winning strategy. This is a contradiction to the rule of the game.
		
		Consider the increasing enumeration $\{r_n : n \in \omega\}$ of $\{ r \in \omega : i_r = 1\}$.
		For each $n \in \omega$, we have $m_n \in \omega$ satisfying for every $\tau \in T$ extending $\sigma$, we have $[\tau(r_n) \ge m_n \implies \tau(r_n+1) = 1]$.
		Fix $f \in \omega^\omega$.
		Consider the play of Player I that plays $\max\{m_n, f(n)\}$ at stage $r_n / 2$.
		Since Player II wins, there is $g \in \scrA$ such that
		\[
		\max\{m_n, f(n)\} \le g(r_n / 2).
		\]
		So $\scrA' = \{ \seq{g(r_n / 2) : n \in \omega} : g \in \scrA \}$ is a dominating family.
		We have $\abs{\scrA} \ge \frakd$.

		We next consider the next case:
		\begin{itemize}
			\item (Case 2) For every $\sigma \in T$ of even length, there is an even number $r \ge \abs{\sigma}$ such that for every $i \in 2$, there exist infinitely many $m$ and there is $\tau \in T$ extending $\sigma$ such that $[\tau(r) = m \land \tau(r+1) = i]$.
		\end{itemize}
		
		In this case, we can construct a perfect subtree of $T$ and each distinct path of this subtree gives distinct element of $\mathcal{A}$.
		
		In detail, we construct $r_s$, $\sigma_s$ and $m^s_0 < m^s_1$ for $s \in 2^{<^\omega}$ such that $\sigma_{s \append i}(r_s) = m^s_i$, $\sigma_{s \append i}(r_s + 1) = i$ for every $i < 2$.
		For each $f \in 2^\omega$, put $\sigma_f = \bigcup_{n \in \omega} \sigma_{f \upharpoonright n}$.
		Since II wins, we can take $g_f \in \scrA$ that witnesses $\sigma_f$ is a winning play.
		Take distinct $f$ and $f'$ in $2^\omega$.
		Let $\Delta = \min \{ n : f(n) \ne f'(n) \}$ and $s = f \upharpoonright \Delta = f' \upharpoonright \Delta$.
		We may assume that  $f(\Delta) = 0$ and $f'(\Delta) = 1$.
		We have $\sigma_f(r_s) = m^s_0, \sigma_f(r_s+1) = 0, \sigma_{f'}(r_s) = m^s_1$ and $\sigma_{f'}(r_s+1) = 1$.
		Then by the rule of the game, we have
		\[
		g_f(r_s / 2) \le m^s_0 < m^s_1 < g_{f'}(r_s / 2).
		\]
		So we have $g_f \ne g_{f'}$.
		Therefore, we have $\abs{\mathcal{A}} = \frakc$ in this case.
		
		In either case, we have $\abs{\scrA} \ge \frakd$, so we have shown $\bII \ge \frakd$.
	\end{proof}

    Using terminology in \cite[Section 10]{blass2010combinatorial}, $\bIstar$ is equal to the global, adaptive, finite prediction version of the evasion number. Moreover, in the article it was shown that this invariant is equal to $\frakb$. So we have $\bIstar = \frakb$. But for the sake of completeness, we include the proof.

    \begin{thm}
        $\bIstar = \frakb$ holds.
    \end{thm}
    \begin{proof}
        It is clear that $\frakb \le \bIstar$. We show $\bIstar \le \frakb$.

        Take an unbounded family $\mathcal{A}$ of $\omega^\omega$.
        Take an arbitrary strategy $\sigma \colon \omega^{<\omega} \to \omega$ of Player I.
        We have to show that $\sigma$ is not a winning strategy for the bounding* game with respect to $\mathcal{A}$.

        Fix an enumeration $\seq{s_i : i \in \omega}$ of $\omega^{<\omega}$ that satisfies $\abs{s_i} \le i$ for every $i$.
        For each $s \in \omega^{<\omega}$ and $n \in \omega \setminus \abs{s}$, we put
        \[
        \sigma_s(n) = \max \{ x(n) : s \subseteq x \in \omega^\omega \text{ and } (\forall k \ge \abs{s})(x(k) \le \sigma(x \upharpoonright k) \}.
        \]
        It can be easily checked that $\sigma_s(n)$ is in $\omega$.
        We define $f$ by
        \[
        f(n) = \max (\{ \sigma_{s_i}(n) : i < n \} \cup \{ 0 \}).
        \]
        Take $g \in \scrA$ that is not dominated by $f$.
        Consider the play in which Player I obeys the strategy $\sigma$ and Player II plays $g$.
        Suppose that Player I wins. Then there is $n_0 \in \omega$ such that $(\forall n \ge n_0)(g(n) \le \sigma(g \upharpoonright n))$.
        Take $m_0 \in \omega$ such that $s_{m_0} = g \upharpoonright n_0$.
        Then we have for every $m > m_0$:
        \[
        g(m) \le \sigma_{s_{m_0}}(m) \le f(m).
        \]
        This means that $f$ dominates $g$, which is a contradiction. 
    \end{proof}
	
    \begin{thm}\label{thm:bIIstar}
        $\bIIstar = \frakc$ holds.
    \end{thm}
    \begin{proof}
		Fix $\scrA \subseteq \omega^\omega$ such that Player II has a winning strategy $\tau$ for the bounding* game with respect to $\scrA$. We shall show that $\scrA$ is of size $\frakc$.
		Consider the game tree $T \subseteq \omega^{<\omega}$ that the strategy determines.
		
		First, assume the following.
		\begin{itemize}
			\item (Case 1) There is a $\sigma \in T$ such that for every odd $k \ge |\sigma|$, there is an $m_k < \omega$ such that for every $\tau \in T$ extending $\sigma$ with $|\tau| > k$, we have $\tau(k) = m_k$.
		\end{itemize}
		
		Fix the witness $\sigma, \seq{m_k : k \ge |\sigma|}$ for Case 1.
		
		
		Consider the next play.
		\begin{table}[H]
			\centering
			\begin{tabular}{l|llllllllllll}
				Player I  & $\sigma(0)$ &              & $\dots$ & $\sigma(|\sigma|-2)$ &               & $m_{|\sigma|}$ &       & $m_{|\sigma| + 2}$ &       & $\dots$ \\ \hline
				Player II &             & $\sigma(1)$ &       $\dots$      &               & $\sigma(|\sigma|-1)$ &       & $m_{|\sigma|}$ &       & $m_{|\sigma| + 2}$ &        
			\end{tabular}
		\end{table}
		
		Then the sequence defined by the play of Player II does not dominate that defined by the play of Player I. So Player II loses. This is a contradiction.
		
		So Case 1 is false.
		Thus we have
		\begin{itemize}
			\item (Case 2) For every $\sigma \in T$, there is an odd number $k \ge |\sigma|$ such that for every $m < \omega$, there is $\tau \in T$ extending $\sigma$ with $|\tau| > k$ such that $\tau(k) \neq m$.
		\end{itemize}
        Note that there are $\tau_0, \tau_1 \supseteq \sigma$ with $|\tau_0|, |\tau_1| > k$ such that $\tau_0(k) \neq \tau_1(k)$ in Case 2.
  
		Now we can construct a subtree of $T$ in the following manner.
		First we put $\sigma_\varnothing = \varnothing$.
		Suppose we have $\seq{\sigma_s : s \in 2^{\le l}}$.
		Then for each $s \in 2^l$, we can take $\sigma_{s \append 0}, \sigma_{s \append 1} \supseteq \sigma_s$ and $k_s \ge |\sigma_s|$ such that $\sigma_{s \append 0}(k_s) \neq \sigma_{s \append 1}(k_s)$.
		
		Now for each $f \in 2^\omega$, we put $\sigma_f$ by $\sigma_f = \bigcup_{n \in \omega} \sigma_{f \upharpoonright n}$.
		
		For each $f \in 2^\omega$, we have $\sigma_f \in [T]$. So Player II wins at the play $\sigma_f$.
		So by the definition of the game, we can take $x_f \in \scrA$ such that $x_f(k) = \sigma_f(2k + 1)$.
		
		We now claim that if $f$ and $g$ are distinct elements of $2^\omega$, then we have $x_f \ne x_g$.
		Let $n := \min \{ n' : f(n') \ne g(n') \}$. Put $s = f \upharpoonright n = g \upharpoonright n$.
		We may assume that $f(n) = 0$ and $g(n) = 1$.
		We have the following:

        \[ x_f\left( \dfrac{k_s -1}{2} \right) = \sigma_f(k_s) = \sigma_{s \append 0}(k_s) \neq \sigma_{s \append 1}(k_s) = \sigma_g(k_s) = x_g\left( \dfrac{k_s -1}{2} \right). \]
		So we have $x_f \ne x_g$.
		
		Therefore we have $\abs{\scrA} \ge \abs{\{x_f : f \in 2^\omega\}} = \frakc$.
	\end{proof}

	\section{Dominating games}
 
    In this section, we consider games related to  dominating families.
	
	Fix a set $\scrA \subseteq \omega^\omega$.
	We call the following game the \textit{dominating game} with respect to $\scrA$:
	
	\begin{table}[H]
		\centering
		\begin{tabular}{l|llllll}
			Player I  & $n_0$ &       & $n_1$ &       & $\dots$ &     \\  \hline
			Player II &       & $i_0$ &       & $i_1$ &         & $\dots$
		\end{tabular}
	\end{table}
	
	Here, $\seq{n_k : k  \in \omega}$ is a sequence of numbers in $\omega$ and $\seq{i_k : k  \in \omega}$ is a sequence of numbers in $2$.
	Player II wins when Player II played $1$ eventually and there is $g \in \scrA$ such that 
	\begin{align*}
		\{ k \in \omega :  i_k = 1 \} = \{ k \in \omega : n_k < g(k) \}.
	\end{align*}

	We call the following game the \textit{dominating* game} with respect to $\scrA$:
	
	\begin{table}[H]
		\centering
		\begin{tabular}{l|llllll}
			Player I  & $n_0$ &       & $n_1$ &       & $\dots$ &     \\  \hline
			Player II &       & $m_0$ &       & $m_1$ &         & $\dots$
		\end{tabular}
	\end{table}
 
	Here, $\seq{n_k : k  \in \omega}$ and $\seq{m_k : k  \in \omega}$ are sequences of numbers in $\omega$.
    Player II wins when
	\begin{align*}
        \seq{m_k : k  \in \omega} \in \mathcal{A} \text{ and } (\forall^\infty k)(n_k < m_k).
    \end{align*}
	
	We define $\dI, \dII, \dIstar$ and $\dIIstar$ by using dominating games and dominating* games in the same fashion as Definition \ref{defi:bIetc}.

    \begin{thm}
        $\dI = \dII = \dIstar = \frakd$ and $\dIIstar = \frakc$ hold.
    \end{thm}
    \begin{proof}
        $\frakd \le \dI$ is easy.
        $\dII \le \frakd$ follows from the observation that for every totally dominating family $\scrA$, Player II has a winning strategy for the dominating game with respect to $\scrA$.
        So we have $\dI = \dII = \frakd$.

        $\dIIstar = \frakc$ follows from $\bIIstar = \frakc$ which was shown in Theorem \ref{thm:bIIstar}, since the dominating* game is harder for Player II than the bounding* game.

        We know $\frakd = \dI \le \dIstar$. So the remaining work is to show $\dIstar \le \frakd$.
        To show it, let $\pi \colon \omega \to \omega^{<\omega}$ be a bijection.
        Fix a dominating family $\mathcal{F} \subseteq \omega^\omega$.
        For $g \in \mathcal{F}$, we define $g' \in \omega^\omega$ so that 
        \[
        (\forall n)((g \circ \pi^{-1})(g' \upharpoonright n) < g'(n)).
        \]
        This $g'$ can be constructed by induction on $n$.
        Put
        \[
        \scrA = \{ g' : g \in \mathcal{F} \}.
        \]
        Take an arbitrary strategy $\sigma \colon \omega^{<\omega} \to \omega$ of Player I. We have to show that $\sigma$ is not a winning strategy.
        Since $\mathcal{F}$ is a dominating family, we can take $g \in \mathcal{F}$ that dominates $\sigma \circ \pi$. Then for all but finitely many $m$, we have
        \[
        \sigma(g' \upharpoonright n) = \sigma(\pi(\pi^{-1}(g' \upharpoonright n))) \le g(\pi^{-1}(g' \upharpoonright n)) \le g'(n).
        \]
        This inequality means if Player II plays $g'$, then Player II wins against Player I who obeys the strategy $\sigma$.
        So we have proved $\sigma$ is not a winning strategy.
    \end{proof}
 
	\section{Reaping games}

    In this section, we consider reaping games and reaping* games, which are related to reaping families.
    The main result of this section is that $\max \{ \frakr, \frakd \} \le \rI \le \max \{ \frakr_\sigma, \frakd \}$. Of course, if $\mathfrak{r}=\mathfrak{r}_\sigma$ it would turn out that $\max\{\frakr,\frakd\}=\rI$. It is in fact an open question if it is consistent that $\frakr\not=\frakr_\sigma$.  For more information the reader may want to look at \cite{splittingreaping}.
    
	Fix a set $\scrA \subseteq [\omega]^\omega$.
	We call the following game the \textit{reaping game} with respect to $\scrA$:
	
	\begin{table}[H]
		\centering
		\begin{tabular}{l|llllll}
			Player I  & $n_0$ &       & $n_1$ &       & $\dots$ &     \\  \hline
			Player II &       & $i_0$ &       & $i_1$ &         & $\dots$
		\end{tabular}
	\end{table}
 
 	Here, $n_0 < n_1 < n_2 < \dots < n_k < \dots$ are increasing numbers in $\omega$, $i_k$ ($k \in \omega$) are elements in $\{0,1\}$.
	Player II wins when there is $A \in \scrA$ such that 
	\begin{align*}
		\{ n_k : k \in \omega \} \cap A = \{ n_k : k \in \omega \text{ and } i_k = 1\} \text{ and } A \text{ reaps } \{ n_k : k \in \omega \}.
	\end{align*}
 
	We call the following game the \textit{reaping* game} with respect to $\scrA$:
	
	\begin{table}[h]
		\centering
		\begin{tabular}{l|llllll}
			Player I  & $i_0$ &       & $i_1$ &       & $\dots$ &     \\  \hline
			Player II &       & $j_0$ &       & $j_1$ &         & $\dots$
		\end{tabular}
	\end{table}
	
	Here, $\seq{i_k : k  \in \omega}$ and $\seq{j_k : k  \in \omega}$ are sequences of elements in $\{0,1\}$.
	Player II wins when Player II played $1$ infinitely often and 
	\begin{align*}
    \{ k \in \omega : j_k = 1 \} \in \scrA \text{ and  reaps } \{ k \in \omega : i_k = 1\}. 
	\end{align*}
	
	We define $\rI, \rII, \rIstar$ and $\rIIstar$ using reaping games and reaping* games in the same fashion as Definition \ref{defi:bIetc}.

    Note that, when Player II wins in the reaping game, then the set $A \in \scrA$ that witnesses Player II wins satisfies the following condition.
    \begin{enumerate}
        \item if $\seq{i_k : k \in \omega}$ is eventually zero, $A$ is almost disjoint from $\seq{n_k : k \in \omega}$.
        \item if the digit 1 appears infinitely often in $\seq{i_k : k \in \omega}$, $A$ is almost contained in $\seq{n_k : k \in \omega}$ and $A =^* \{ n_k : i_k = 1 \}$.
    \end{enumerate}
 
	\begin{thm}
		$\rII = \frakc$ holds.
	\end{thm}
    \begin{proof}
		Fix $\scrA \subseteq [\omega]^\omega$ such that Player II has a winning strategy for the reaping game with respect to $\scrA$. Fix such a strategy. We shall show that $\scrA$ is of size $\frakc$.
		Consider the game tree $T \subseteq \omega^{<\omega}$ that the strategy determines.
  
		First, assume the following.
		\begin{itemize}
			\item (Case 1) There is a  $\sigma \in T$ of even length such that for every $m > \sigma(\abs{\sigma}-2)$ there is $i_m < 2$ such that for every $\tau \in T$ extending $\sigma$ and every $r \in [\abs{\sigma}, \abs{\tau})$, $\tau(r) = m$ implies $\tau(r+1) = i_m$.
		\end{itemize}

        Fix a witness $\sigma$ and $\seq{i_m : m \ge \sigma(\abs{\sigma}-2)}$ of Case 1.
        If $i_m$'s are eventually zero, clearly there is a play that Player II loses along the strategy, which is a contradiction.
        
        So $i_m$'s are not eventually zero. Take an infinite set $X \subseteq [\sigma(\abs{\sigma}-2), \omega)$ such that $i_m = 1$ for every $m \in X$.
        Considering Player I plays an arbitrary subset of $X$, Player II must accordingly produce an $A \in \mathcal{A}$ that is almost equal to this set.
        But the cardinality of $[X]^{\omega} / \mathsf{fin}$ is $\frakc$. So we have shown $\abs{\scrA} = \frakc$.

        Next, we assume the negation of Case 1.
        In this case we can construct a perfect subtree of $T$ whose different paths yield different members of $\scrA$.
    \end{proof}
	
	\begin{thm}
		$\rI \ge \frakr, \frakd$ holds.
	\end{thm}
    \begin{proof}
        That $\rI \ge \frakr$ is easy.
        We show $\rI \ge \frakd$.
        Fix a family $\mathcal{A}$ such that Player I has no winning strategy for the reaping game with respect to $\scrA$.
        For $A \in [\omega]^\omega$, let $e_A$ be the increasing enumeration of $A$.
        Put
        \[
        \mathcal{F} = \{ e_B : B \text{ is almost equal to some } A \in \scrA \}.
        \]
        Then we have $\abs{\mathcal{F}} = \abs{\scrA}$.
        
        We shall show that $\mathcal{F}$ is a dominating family.
        So we fix an arbitrary increasing function $g \in \omega^\omega$.
        Let us consider the following strategy of Player I:
        First play $f(0)$.
        If Player II responds $0$ then play $f(0)+1$, otherwise play $f(1)$.
        In general, if in the last time Player I played $f(l)+m$ and Player II responded $0$, then play $f(l)+m+1$, otherwise, play $f(l+1)+m$.

        By the assumption, this strategy is not a winning strategy, so there is a play of Player II $\bar{i} \in 2^\omega$ and $A \in \scrA$ such that $A$ witnesses Player II wins with $\bar{i}$ against the strategy.

        Let $\seq{n_k : k \in \omega}$ be the corresponding play of Player I.
        If $\bar{i}$ is eventually zero, then $\seq{n_k : k \in \omega}$ contains almost all integers in $\omega$.
        Moreover, by the rule of the game, $A$ is almost disjoint from this set. This cannot happen.

        So the digit $1$ appears infinitely often in $\bar{i}$.
        Then $A =^* \{ n_k : i_k = 1 \}$.
        Call the last set $B$.
        Then $e_B \in \mathcal{F}$ and $e_B$ dominates $f$ by the choice of the strategy.

        Therefore, $\mathcal{F}$ is a dominating family.
    \end{proof}

    Define a cardinal invariant $\frakr_\mathrm{simult}$ as follows:
    \begin{align*}
    \frakr_\mathrm{simult} &= \min \{ \mathcal{F} \subseteq ([\omega]^\omega)^\omega : (\forall \seq{A_n \in [\omega]^\omega : n \in \omega})(\exists \seq{B_n : n \in \omega} \in \mathcal{F}) \\
    &\hspace{5cm} [(\exists n)(B_0 \subseteq \omega \setminus A_n) \text{ or } (\forall n)(B_n \subseteq A_n)] \} 
    \end{align*}

    It can be easily seen that $\frakr, \frakd \le \frakr_\mathrm{simult}$.

    \begin{prop}
        $\frakr_\mathrm{simult} \le \max \{ \frakr_\sigma, \frakd \}$.
    \end{prop}
    \begin{proof}
        Let $\mathcal{R}$ be a $\sigma$-reaping family of size $\frakr_\sigma$ and $\mathcal{D}$ be a totally dominating family of $\omega^\omega$ of size $\frakd$.
        For $(C, h) \in \mathcal{R} \times \mathcal{D}$, we let
        \[
        B^{C,h}_n = C \setminus h(n).
        \]

        We now show $\{ \seq{B^{C,h}_n : n \in \omega} : (C, h) \in \mathcal{R} \times \mathcal{D} \}$ is a witness of $\frakr_\mathrm{simult}$.
        Fix a sequence $\seq{A_n \in [\omega]^\omega : n \in \omega}$.
        Since $\mathcal{R}$ is a $\sigma$-reaping family, we can take $C \in \mathcal{R}$ such that
        \[
        (\forall n)(C \subseteq^* A_n \text{ or } C \subseteq^* \omega \setminus A_n).
        \]
        We first consider the case $C \subseteq^* \omega \setminus A_n$ for some $n$.
        Take $m$ such that $C \setminus m \subseteq \omega \setminus A_n$.
        Take $h \in \mathcal{D}$ such that $h(0) \ge m$.
        Then clearly, $\seq{B^{C,h}_n : n \in \omega}$ satisfies the condition of $\frakr_\mathrm{simult}$.

        We next consider the case $C \not \subseteq^* \omega \setminus A_n$ for every $n$. Then for every $n$, we have $C \subseteq^* A_n$.
        Let $f \in \omega^\omega$ be such that $C \setminus f(n) \subseteq A_n$.
        Take $h \in \mathcal{D}$ that totally dominates $f$.
        Then we also have $C \setminus h(n) \subseteq A_n$.
        Then $\seq{B^{C,h}_n : n \in \omega}$ satisfies the condition of $\frakr_\mathrm{simult}$.
    \end{proof}

    \begin{thm}
        $\rI \le \frakr_\mathrm{simult}$ holds.
    \end{thm}
    \begin{proof}
        Fix a witness $\mathcal{F}$ of $\frakr_\mathrm{simult}$.

        Using a bijection between $\omega$ and $\omega^{<\omega}$, we think $\mathcal{F}$ is a subset of $([\omega]^\omega)^{(\omega^{<\omega})}$. That is, $\mathcal{F}$ satisfies the following condition:
        \begin{align*}
        &(\forall \seq{A_t \in [\omega]^\omega : t \in \omega^{<\omega}})(\exists \seq{B_t : t \in \omega^{<\omega}} \in \mathcal{F}) \\
        &\hspace{3cm} [(\exists t)(B_\varnothing \subseteq \omega \setminus A_t) \text{ or } (\forall t)(B_t \subseteq A_t)].  \label{eq:simult} \tag{$*$}
        \end{align*}

        Fix $\bar{B} = \seq{B_t : t \in \omega^{<\omega}} \in \mathcal{F}$.
        We define $\seq{b^{\bar{B}}_n : n \in \omega}$ by
        \begin{align*}
            b^{\bar{B}}_0 &= \varnothing, \\
            b^{\bar{B}}_{n+1} &= b^{\bar{B}}_n \append \seq{\min B_{b^{\bar{B}}_n}}.
        \end{align*}
        Put $\varphi(\bar{B}) = \range \bigcup_n b^{\bar{B}}_n$.

        Define $\mathcal{A}$ by
        \[
        \mathcal{A} = \{ \varphi(\bar{B}) : \bar{B} \in \mathcal{F} \} \cup \{ X : \bar{B} \in \mathcal{F}, \text{$X$ and $B_0$ are almost equal} \}.
        \]
        Note that $\abs{\mathcal{A}} \le \abs{\mathcal{F}}$.

        We show that Player I has no winning strategy for the reaping game with respect to $\mathcal{A}$.

        Let $\sigma \colon 2^{<\omega} \to \omega$ be an arbitrary strategy of Player I.
        Consider the tree $T \subseteq \omega^\omega$ defined as follows.
        $T \cap \omega^1 = \{ \seq{\sigma(\varnothing)}, \seq{\sigma(0)}, , \seq{\sigma(00)}, \dots \}$.
        In general, the node whose label is $\sigma(s)$ has children whose labels are $\sigma(s \append \seq{1} \append \seq{0}^m)$ for $m \in \omega$.

        We now put for each $t \in \omega^{<\omega}$
        \[
        A_t = \begin{cases} \suc_{T_\sigma}(t) & \text{(if $t \in T$)} \\ \omega & \text{(otherwise)}. \end{cases}
        \]
        Then applying (\ref{eq:simult}), we can take $\seq{B_t : t \in \omega^{<\omega}} \in \mathcal{F}$ such that
        \[
        (\exists t)(B_\varnothing \subseteq \omega \setminus A_t) \text{ or } (\forall t)(B_t \subseteq A_t).
        \]
        
        Consider the former case: $(\exists t)(B_\varnothing \subseteq \omega \setminus A_t)$.
        Fix such a $t$. Then $t$ must be in $T$.
        If $t = \varnothing$, consider the following play:

    	\begin{table}[H]
    		\centering
    		\begin{tabular}{l|llllllll}
    			Player I & $\cdot$ & & $\cdot$ & & $\cdot$ & & $\dots$ & \\  \hline
    			Player II & & $0$ & & $0$ & & $0$ & & $\dots$
    		\end{tabular}
    	\end{table}
        The middle dots ($\cdot$) in the first row mean the play along $\sigma$.
        Then $B_\varnothing$, which is in $\mathcal{A}$ is a witness that Player II wins. Indeed, if $\seq{n_k : k \in \omega}$ is the play of Player I, then $B_\varnothing \subseteq \omega \setminus \{ n_k : k \in \omega \}$.

        If $t \ne \varnothing$, take $s \in 2^{<\omega}$ such that the label of $t$ is $\varphi(s)$.
        Consider the following play:

    	\begin{table}[H]
    		\centering
    		\begin{tabular}{l|lllllll|lllllll}
    			Player I & $\cdot$ & & $\cdot$ & & $\dots$ & $\cdot$ & & $\cdot$ & & $\cdot$ & & $\cdot$ & & $\dots$ \\  \hline
    			Player II & & $s(0)$ & & $s(1)$ & $\dots$ & & $s(\abs{s}-1)$ & & $1$ & & $0$ & & $0$ & $\dots$
    		\end{tabular}
    	\end{table}
     
        Then a real almost equal to $B_\varnothing$, which is in $\mathcal{A}$ is a witness that Player II wins.
        
        Consider the latter case $(\forall t)(B_t \subseteq A_t)$.
        Let $A = \varphi(\seq{B_t : t \in \omega^{<\omega}})$.
        Enumerate $A$ by $A = \{a_n : n \in \omega\}$ in ascending order.
        Take the unique $m_0$ such that $a_0 = \sigma(\seq{0}^{m_0})$ and put $s_0 = \seq{0}^{m_0} \append \seq{1}$.
        By induction on $k$, take the unique $m_k$ such that $a_k = \sigma(s_k \append \seq{0}^{m_k})$ and put $s_{k+1} = s_k \append \seq{0}^{m_k} \append \seq{1}$.
        Put $\bar{i} = \bigcup_k s_k$, which is a play of Player II.
        Let $\seq{n_k : k \in \omega}$ be the corresponding play of Player I. That is $n_k = \sigma(\bar{i} \upharpoonright k)$.
        Then we have $A \in \mathcal{A}$ and $A = \{ n_k : k \in \omega, i_k = 1 \}$. So Player II wins.

        Therefore, in either case, Player II wins.
        So $\sigma$ is not a winning strategy.
    \end{proof}

    \begin{question}
        Does $\ZFC$ prove that $\rI$ is equal to $\max \{ \frakr_\sigma, \frakd \}$?
    \end{question}

    Because of the following theorem, the cardinal invariants regarding reaping* games are not worth considering.

    \begin{thm}
        For every $\scrA \subseteq [\omega]^\omega$, Player I has a winning strategy for the reaping* game with respect to $\scrA$.
    \end{thm}
    \begin{proof}
        Consider the following strategy of Player I:
        \begin{itemize}
            \item Play $0$ first.
            \item If the previous play of Player II is $1$, change the move from the previous play of Player I.
            \item Otherwise, play the same move as the previous play of Player I.
        \end{itemize}
        It can be easily seen that this is a winning strategy.
    \end{proof}
 
	\section{Anti-localizing games}

    In this section, we consider games related to the cardinal invariant $\add(\nul)$.

    Let $\mathcal{C} = \{ \varphi : \varphi \text{ is a function with domain } \omega \text{ that satisfies } \varphi(n) \in [\omega]^{n+1} \text{ for every } n \in \omega \} $.
    We call elements in $\mathcal{C}$ slaloms.
	
	Fix a set $\scrA \subseteq \omega^\omega$.
	We call the following game the \textit{anti-localizing game} with respect to $\scrA$:
	
	\begin{table}[H]
		\centering
		\begin{tabular}{l|llllll}
			Player I  & $a_0$ &       & $a_1$ &       & $\dots$ &     \\  \hline
			Player II &       & $i_0$ &       & $i_1$ &         & $\dots$
		\end{tabular}
	\end{table}
	
	Here, $\seq{a_k : k  \in \omega}$ is a sequence with $a_k \in [\omega]^{k+1}$ for every $k$ and $\seq{i_k : k  \in \omega}$ is a sequence of numbers in $2$.
	Player II wins when Player II played $1$ infinitely often and there is $x \in \scrA$ such that 
	\begin{align*}
		\{ k \in \omega :  i_k = 1 \} = \{ k \in \omega : x(k) \not \in a_k \}.
	\end{align*}

	We call the following game the \textit{anti-localizing* game} with respect to $\scrA$:
	
	\begin{table}[H]
		\centering
		\begin{tabular}{l|llllll}
			Player I  & $a_0$ &       & $a_1$ &       & $\dots$ &     \\  \hline
			Player II &       & $n_0$ &       & $n_1$ &         & $\dots$
		\end{tabular}
	\end{table}
 
	Here, $\seq{a_k : k  \in \omega} \in \mathcal{C}$ and $\seq{n_k : k  \in \omega}$ is a sequence of numbers in $\omega$.
    Player II wins when
	\begin{align*}
        \seq{n_k : k  \in \omega} \in \mathcal{A} \text{ and } (\exists^\infty k)(n_k \notin a_k).
    \end{align*}
	
	We define $\aNI, \aNII, \aNIstar$ and $\aNIIstar$ using anti-localizing games and anti-localizing*-games in the same fashion as Definition \ref{defi:bIetc}.
	
	\begin{thm}\label{thm:ani}
		$\aNI = \add(\nul)$ holds.
	\end{thm}
    Before proving this theorem, we recall the relationship between $\add(\nul)$ and slaloms.

    \begin{fact}[{{\cite[Theorem 4.11]{bartoszynski2010invariants}}}]\label{fact:addnull}
        The following are equivalent.
        \begin{enumerate}
            \item $\add(\nul) \le \kappa$.
            \item There is a family $\scrA \subseteq \omega^\omega$ of size $\le \kappa$ such that $(\forall \varphi \in \mathcal{C})(\exists x \in \scrA)(\exists^\infty n)( x(n) \not \in \varphi(n))$ holds.
            \item There is a family $\scrA \subseteq \omega^\omega$ of size $\le \kappa$ such that $(\forall f \in \mathcal{C}^\omega)(\exists x \in \scrA)(\forall m)(\exists^\infty n)(x(n) \not \in f(m)(n))$ holds.
        \end{enumerate}
    \end{fact}
    
    \begin{proof}[Proof of Theorem \ref{thm:ani}]
        $\aNI \ge \add(\nul)$ holds by Fact \ref{fact:addnull}.
        We prove $\aNI \le \add(\nul)$.
        Take a witness $\scrA$ for (3) of Fact \ref{fact:addnull}.
        Now we want to prove that Player I has no winning strategy for the anti-localizing game with respect to $\scrA$.
        Take a strategy $\sigma \colon 2^{<\omega} \to [\omega]^{<\omega}$ of Player I. 
        Since $\scrA$ satisfies the condition in (3) of Fact \ref{fact:addnull}, we can take $x \in \scrA$ such that $(\exists^\infty n)(x(k) \not \in \sigma(\bar{i} \upharpoonright k)))$ for every $\bar{i} \in \zero$.

		We now put $\bar{i} \in 2^\omega$ by
		\[
		i_k = \begin{cases} 1 & \text{(if $x(i) \not \in \sigma(\bar{i} \upharpoonright k)$)} \\ 0 & \text{(otherwise)} \end{cases}
		\]
		If $\bar{i} \in \zero$, then $(\exists^\infty n)(x(k) \not \in \sigma(\bar{i} \upharpoonright k))$ by the choice of $x$.
		But this fact and the choice of $\bar{i}$ imply $\bar{i} \not \in \zero$. It's a contradiction. So $\bar{i} \not \in \zero$. Therefore $\bar{i}$ is a play of Player II that wins against the strategy $\sigma$ of Player I.
    \end{proof}
	
	\begin{thm}\label{thm:aNII}
		$\aNII = \cov(\meager)$ holds.
	\end{thm}

    Before proving this theorem, we recall the relationship between $\cov(\meager)$ and slaloms.
 
    \begin{fact}[{{\cite[Lemma 2.4.2]{bartoszynski1995set}}}]\label{fact:covmeager}
        The following are equivalent.
        \begin{enumerate}
            \item $\cov(\meager) \le \kappa$.
            \item \label{item:totalavoid} There is a family $\scrA \subseteq \omega^\omega$ of size $\le \kappa$ such that $(\forall \varphi \in \mathcal{C})(\exists x \in \scrA)(\forall n)( x(n) \not \in \varphi(n))$ holds.
        \end{enumerate}
    \end{fact}

    In addition, the following characterization is well-known.

    \begin{fact}\label{fact:martinsaxiom}
        The following are equivalent.
        \begin{enumerate}
            \item $\kappa < \cov(\meager)$.
            \item Martin's axiom for countable posets with $\kappa$-many dense subsets.
        \end{enumerate}
    \end{fact}
    
	\begin{proof}[Proof of Theorem \ref{thm:aNII}]
		We first prove $\aNII \le \cov(\meager)$.
        Take a family $\scrA \subseteq \omega^\omega$ of size $\cov(\meager)$ that satisfies (\ref{item:totalavoid}) of Fact \ref{fact:covmeager}. 
		Then the strategy that plays $1$ always is a winning strategy for Player II.

        We next prove $\cov(\meager) \le \aNII$.
        Assuming $\kappa < \cov(\meager)$, we shall prove $\kappa < \aNII$.
        Fix a family $\scrA$ of size $\kappa$.
        Take an arbitrary strategy $\tau$ of Player II.
        We show that $\tau$ is not a winning strategy.

        We may assume that Player II plays the digit $1$ infinitely often along $\tau$, otherwise, $\tau$ is clearly not a winning strategy.

        Set $P = \bigcup_n \prod_{i < n} [\omega]^{i+1}$.
        For each $x \in \scrA$, we define a set $D_x$ as follows:
        \[
        D_x = \{ p \in P : (\exists k \in \dom(p))(x(k) \in p(k) \text{ and } \tau(p \upharpoonright (k+1)) = 1 \}.
        \]
        Then each $D_x$ is a dense subset of $P$, using the above assumption.

        Therefore, by Fact \ref{fact:martinsaxiom}, we can take a filter $G \subseteq P$ that intersects with all $D_x$'s. Put $g = \bigcup G$. Then if Player I plays $g$, then Player I wins against Player II, who obeys the strategy $\tau$.
    \end{proof}

    \begin{thm}
        $\aNIstar = \add(\nul)$ holds.
    \end{thm}
    \begin{proof}
        Using terminology in \cite[Section 10]{blass2010combinatorial}, $\aNIstar$ is equal to the global, adaptive, prediction specified by the predefined function version of evasion number. Moreover, in the article, it was shown that this invariant is equal to $\add(\nul)$.
    \end{proof}

    \begin{thm}
        $\aNIIstar = \frakc$ holds.
    \end{thm}
	\begin{proof}
		Fix $\scrA \subseteq \omega^\omega$ such that Player II has a winning strategy $\tau$ for the anti-localizing* game with respect to $\scrA$. We shall show that $\scrA$ is of size $\frakc$.
		Consider the game tree $T \subseteq \prod_{n < \omega} X(n)$ that $\tau$ determines, where $X(2n) = [ \omega ]^{< n + 1}$ and $X(2n + 1) = \omega$.
		
		First, assume the following.
		\begin{itemize}
			\item (Case 1) There is a $\sigma \in T$ such that for every odd $k \ge |\sigma|$, there is an $n_k < \omega$ such that for every $\tau \in T$ extending $\sigma$ with $|\tau| > k$, we have $\tau(k) = m_k$.
		\end{itemize}
		
		Fix the witness $\sigma, \seq{n_k : k \ge |\sigma|}$ for Case 1.
		
		Consider the next play.
		\begin{table}[H]
			\centering
			\begin{tabular}{l|llllllllllll}
				Player I  & $\sigma(0)$ &              & $\dots$ & $\sigma(|\sigma|-2)$ &               & $\{ n_{|\sigma|} \}$ &       & $\{ n_{|\sigma| + 2} \}$ &       & $\dots$ \\ \hline
				Player II &             & $\sigma(1)$ &       $\dots$    &                & $\sigma(|\sigma|-1)$ &       & $n_{|\sigma|}$ &       & $n_{|\sigma| + 2}$ &        
			\end{tabular}
		\end{table}
		
		Then the sequence defined by the play of Player II does not avoid the slalom defined by the play of Player I. So Player II loses. This is a contradiction.
		
		So Case 1 is false.
		Thus we have
		\begin{itemize}
			\item (Case 2) For every $\sigma \in T$, there is an odd number $k \ge |\sigma|$ such that for every $n < \omega$, there is $\tau \in T$ extending $\sigma$ with $|\tau| > k$ such that $\tau(k) \neq n$.
		\end{itemize}
        Note that there are $\tau_0, \tau_1 \supseteq \sigma$ with $|\tau_0|, |\tau_1| > k$ such that $\tau_0(k) \neq \tau_1(k)$ in Case 2.
  
		Now we can construct a subtree of $T$ in the following manner.
		First we put $\sigma_\varnothing = \varnothing$.
		Suppose we have $\seq{\sigma_s : s \in 2^{\le l}}$.
		Then for each $s \in 2^l$, we can take $\sigma_{s \append 0}, \sigma_{s \append 1} \supseteq \sigma_s$ and $k_s \ge |\sigma_s|$ such that $\sigma_{s \append 0}(k_s) \neq \sigma_{s \append 1}(k_s)$.
		
		Now for each $f \in 2^\omega$, we put $\sigma_f$ by $\sigma_f = \bigcup_{n \in \omega} \sigma_{f \upharpoonright n}$.
		
		For each $f \in 2^\omega$, we have $\sigma_f \in [T]$. So Player II wins at the play $\sigma_f$.
		So by the definition of the game, we can take $x_f \in \scrA$ such that $x_f(k) = \sigma_f(2k + 1)$.
		  It should be clear that if $f$ and $g$ are distinct elements of $2^\omega$, then we have $x_f \ne x_g$.
		Therefore we have $\abs{\scrA} = \frakc$.
    \end{proof}

    \section{Tallness games}\label{sec:tallness}
    In this section we introduce two game theoretic versions of the tallness property for ideals on countable sets, namely the tallness game and the tallness$^*$ game.
    We show that such games are related to the two well-known ideals $\mathcal{ED}_{\mathrm{fin}}$ and $\mathcal{ED}$ via the Katetov order.
    This will be done by proving that winning strategies in such games correspond either to  winning strategies in the HMM game or to a certain tree related property, both of which were studied by M. Hru\v{s}\'ak, D. Meza-Alc\'antara, and H. Minami in \cite{HMM2010} and by M. Hru\v{s}\'ak in \cite{katetovorderhrusak} respectively. Games related to filters and ideals have also played an important role in the analysis of cardinal invariants of the continuum. An example of this can be found in \cite{MR4270048}, where games related to filters are used to show the consistency of $\omega_1=\mathfrak{u}<\mathfrak{a}.$
    
    Recall that an ideal $\idealI$ on $\omega$ is said to be \textit{tall} if for every infinite set of $\omega$ has an infinite subset in $\idealI$.

    First we recall the definition of the ideals $\mathcal{ED}$ and $\mathcal{ED}_{\mathrm{fin}}$, the Katětov-Blass ordering $\le_{\mathrm{KB}}$ and the uniformity of ideals on $\omega$.

    \begin{defi}
        \begin{enumerate}
            \item $\mathcal{ED}$ is the ideal on $\omega \times \omega$ that is generated by vertical lines and graphs of functions from $\omega$ to $\omega$.
            \item $\mathcal{ED}_{\mathrm{fin}} = \{ A \cap \Delta : A \in \mathcal{ED} \}$, where $\Delta = \{ (m, n) \in \omega^2 : n \le m \}$.
            \item For ideals $\idealI, \idealJ$ on $X, Y$ respectively, we denote $\idealI \le_{\mathrm{K}} \idealJ$ iff there is a  function $f \colon Y \to X$ such that $f^{-1}(I) \in \idealJ$ for every $I \in \idealI$. We call the order $\le_{\mathrm{K}}$ the Katětov ordering.
            \item For ideals $\idealI, \idealJ$ on $X, Y$ respectively, we denote $\idealI \le_{\mathrm{KB}} \idealJ$ iff there is a finite-to-one function $f \colon Y \to X$ such that $f^{-1}(I) \in \idealJ$ for every $I \in \idealI$. We call the order $\le_{\mathrm{K}}$ the Katětov-Blass ordering.
            \item For an ideal on $\omega$, put $\non^*(\idealI) = \min \{ \abs{\mathcal{A}} : \mathcal{A} \subseteq [\omega]^\omega, (\forall I \in \idealI)(\exists A \in \mathcal{A}) \abs{A \cap I} < \aleph_0) \}$.
        \end{enumerate}
    \end{defi}

    Let $\idealI$ be an ideal on $\omega$.
	We call the following game the \textit{tallness game} with respect to $\idealI$:
	
	\begin{table}[H]
		\centering
		\begin{tabular}{l|llllll}
			Player I  & $n_0$ &       & $n_1$ &       & $\dots$ &     \\  \hline
			Player II &       & $i_0$ &       & $i_1$ &         & $\dots$
		\end{tabular}
	\end{table}

	Here, $\seq{n_k : k  \in \omega}$ is an increasing sequence of numbers in $\omega$ and $\seq{i_k : k  \in \omega}$ is a sequence of numbers in $2$.
	Player II wins when
	\begin{align*}
		\{ n_k :  k \in \omega, i_k = 1 \} \in \idealI \cap [\omega]^\omega.
	\end{align*}
The $tallness^*$ \textit{game} is defined in a similar way, but in this case Player II wins when at least one of the following two conditions holds:
\begin{itemize}
    \item $\{ n_k : k\in\omega\}\in \idealI^+$ and $\{ n_k :  k \in \omega, i_k = 1 \} \in \idealI \cap [\omega]^\omega$, or 
    \item $\{n_k : k\in\omega\}\in \idealI$.
\end{itemize}

    We can easily see that if Player I has no winning strategy for the tallness game with respect to $\idealI$, then $\idealI$ is a tall ideal.

    In the following proposition, we fix some bijection between $\omega$ and $\Delta$ and use this bijection implicitly.

    \begin{prop}
        Player II has a winning strategy for the tallness game with respect to $\mathcal{ED}_\mathrm{fin}$.
    \end{prop}
    \begin{proof}
        Consider the following strategy:
        \begin{itemize}
            \item if Player I played $(x, y)$ last and the first coordinate $x$ has appeared so far then return $0$.
            \item Otherwise, return $1$.
        \end{itemize}
        This is a winning strategy since the result is the graph of a single function.
    \end{proof}

    \begin{prop}
        Let $\idealI$ and $\idealJ$ be two ideals and suppose $\idealJ \le_\mathrm{KB} \idealI$.
        Then if Player II has a winning strategy for the tallness game with respect to $\idealJ$, then Player II also has a winning strategy for the tallness game with respect to $\idealI$. 
    \end{prop}
    \begin{proof}
        Let $f \colon \omega \to \omega$ be a witness of $\idealJ \le_\mathrm{KB} \idealI$; a finite-to-one function such that $f^{-1}(J) \in \idealI$ for every $J \in \idealJ$.
        Fix a winning strategy $\tau$ of Player II for the tallness game with respect to $\idealJ$. Then we shall construct a winning strategy of Player II for the tallness game with respect to $\idealI$.

        For an increasing sequence $\bar{m} = \seq{m_0, m_1, \dots, m_k}$ of natural numbers, $\bar{m}^*$ denotes the subsequence of $\bar{m}$ obtained by the following manner:
        if $\bar{m} = \bar{n} \append \seq{m_k}$ and $f(m_k) > f(a)$ for a component $a$ of $\bar{n}^*$, then $\bar{m}^* = \bar{n}^* \append \seq{m_k}$, otherwise $\bar{m}^* = \bar{n}^*$.

        Define a strategy $\tau^*$ of Player II for the tallness game with respect to $\idealI$ as follows: if $\bar{m} = \bar{n} \append \seq{m_k}$ and $\bar{m}^* = \seq{m_{k_0}, m_{k_1}, \dots, m_{k_j}}$.
        \[
        \tau^*(\bar{m}) = \begin{cases} \tau(\seq{f(m_{k_0}), f(m_{k_1}), \dots, f(m_{k_j})}) & \text{(if $\bar{m}^* \ne \bar{n}^*$ holds)} \\ 0 & \text{(if $\bar{m}^* = \bar{n}^*$ holds)} \end{cases}
        \]

        Consider the play $\seq{m_0, m_1, \dots}$ of Player I.
        Let $\seq{m_{k_0}, m_{k_1}, \dots}$ be the subsequence obtained by the above fashion so that $f(m_{k_0}) < f(m_{k_1}) < \dots$.

        In the game with respect to $\idealJ$, Player II wins the following match:
        
         \begin{table}[H]
			\centering
			\begin{tabular}{l|lllll}
				Player I  & $f(m_{k_0})$ & & $f(m_{k_1})$ & & $\dots$
                \\ \hline
				Player II & & $\tau(f(m_{k_0}))$ & & $\tau(f(m_{k_0}), f(m_{k_1}))$ & $\dots$
			\end{tabular}
		\end{table}

        So we have $X := \{ f(m_{k_i}) : \tau(f(m_{k_0}), \dots, f(m_{k_i})) = 1 \} \in \idealJ$.
        By the choice of $f$, we have $f^{-1}(X) \in \idealI$.
        But $f^{-1}(X)$ contains the set $\{ m_{k_i} : \tau(f(m_{k_0}), \dots, f(m_{k_i})) = 1 \}$ and thus the last set is in $\idealI$.
        So $\tau^*$ is a winning strategy.
    \end{proof}

    \begin{cor}\label{cor:winstrofII}
        Let $\idealI$ be an ideal and suppose $\mathcal{ED}_\mathrm{fin} \le_\mathrm{KB} \idealI$. Then Player II has a winning strategy for the tallness game with respect to $\idealI$. \qed
    \end{cor}

    \begin{lem}\label{lem:treeextending}
        Fix an ideal $\idealI$ on $\omega$ and suppose that Player II has a winning strategy for the tallness game with respect to $\idealI$ and fix such a strategy.
        Consider the game tree $T$ that the strategy determines.
        Then, for every $\sigma \in T$, there is $\tau \supseteq \sigma$ of even length such that for every $n > \tau(\abs{\tau}-2)$, we have $\tau \append \seq{n, 1} \in T$.
    \end{lem}
    \begin{proof}
        Let us deny the conclusion.
        Then we have the following.
        \begin{itemize}
            \item There is a $\sigma \in T$ such that for every $\tau \supseteq \sigma$ with even length there is some $n_{\tau} > \tau(\abs{\tau} - 1)$ such that $\tau \append \seq{n_\tau, 0} \in T$.
        \end{itemize}
         
         So we fix the witnesses $\sigma\in T$ and $\{ n_\tau \mid \tau \supseteq \sigma \text{ of even length} \}$. We consider the following play.

         \begin{table}[H]
			\centering
			\begin{tabular}{l|llllllllllll}
				Player I  & $\sigma(0)$ &              & $\dots$ & $\sigma(\abs{\sigma}-2)$ &               & $n_{\sigma}$ &       & $n_{\sigma \append \seq{n_{\sigma}, 0}}$ &       & $\dots$ \\ \hline
				Player II &             & $\sigma(1)$ &       $\dots$      &               & $\sigma(\abs{\sigma}-1)$ &       & $0$ &       & $0$ &        
			\end{tabular}
		\end{table}

        Player II plays 0 all but finitely many, so this is a contradiction.
    \end{proof}

    Let $\idealI$ be an ideal on $\omega$.
	Call the following game the \textit{HMM game} (short for Hrušák--Meza--Minami game) with respect to $\idealI$:
	
	\begin{table}[H]
		\centering
		\begin{tabular}{l|llllll}
			Player I  & $F_0$ &       & $F_1$ &       & $\dots$ &     \\  \hline
			Player II &       & $n_0$ &       & $n_1$ &         & $\dots$
		\end{tabular}
	\end{table}

	Here, $\seq{F_k : k  \in \omega}$ is a sequence of finite sets of elements in $\omega$ and $\seq{n_k : k \in \omega}$ is a sequence of numbers in $\omega$ such that $n_k \not \in F_k$.
	Player I wins when
	\begin{align*}
		\{ n_k : k \in \omega \} \in \mathcal{I}.
	\end{align*}

    This game was invented in \cite{HMM2010}.
    We will show that our tallness games and HMM games are equivalent.
    
    \begin{thm}For an ideal $\mathcal{I}$ in $\omega,$ the following are equivalent:
    \begin{enumerate}
    \item Player I has a winning strategy for the HMM game with respect to $\mathcal{I}$.
    \item Player II has a winning strategy for the tallness game with respect to $\mathcal{I}$.
        \item $\mathcal{ED}_{\mathrm{fin}}\leq_{\mathrm{KB}}\mathcal{I}.$
    \end{enumerate}
    \end{thm}
    
    \begin{proof} First, we prove (1) $\to$ (2).
    Fix a strategy $\tau \colon \omega^{<\omega}\longrightarrow [\omega]^{<\omega}$ for Player I in the HMM game testifying this fact. We define a strategy $\tau^*$ for Player II in tallness game as follows; Before the game starts, Player II will define $\overline{m}_0$ as the empty sequence. Now, suppose that we are on the $k$-th turn and Player II has already defined a finite sequence $\overline{m}_k\in\omega^{<\omega}$. For a given play $n_k$ of Player I, if $n_k\leq \max \tau(\overline{m}_k)$ then Player II answers with $0$ and puts $\overline{m}_{k+1}=\overline{m}$. Otherwise, Player II answers with $1$ and defines $\overline{m}_{k+1}$ as $\overline{m}_{k}^\frown \langle n_k\rangle$.
    
    Consider a match for the tallness game in which Player II played by following $\tau^*$ and let $\langle \overline{m}_k\rangle_{k\in\omega}$ be the associated sequence constructed throughout that match.
    \begin{table}[H]
    		\centering
    		\begin{tabular}{l|llllll}
    			Player I  & $n_0$ &       & $n_1$ &       & $\dots$ &     \\  \hline
    			Player II &       & $i_0$ &       & $i_1$ &         & $\dots$
    		\end{tabular}
    	\end{table}
    From construction, it should be clear that $\langle n_k :  i_k=1\rangle=\overline{m}:=\bigcup\limits_{k\in\omega}\overline{m}_k$. Furthermore, the sequence $\overline{m}$ codes the plays of  Player II in a match of the HMM game in which Player I played by following $\tau$. As $\tau$ is a winning strategy, it must happen that $\{n_k : i_k=1\}\in \mathcal{I}.$
    
    Next, we prove (2) $\to$ (1).
    Fix a game tree $T\subseteq \omega^{<\omega}$ associated with a winning strategy for Player II in the tallness game. We define a strategy for Player II in the HMM game as follows;  Before the game starts, Player I defines $\sigma_{-1}$ as the empty sequence. Now, suppose that we are $k$-th  turn Player I has already defined $\sigma_{k-1}\in T$ to be a sequence of odd length.  Player I now defines $\sigma'_k\in T$ to be a sequence of even length extending $\sigma'_k$ and such that for any $n>\sigma'_k(|\sigma'_k|-2)$ we have that $\sigma'_k {}^\frown \langle n,1\rangle \in T$. This is possible due to Lemma \ref{lem:treeextending}. It is easy to see that such $\sigma'_k$ can be taken in such way that $\sigma'_k(j)=0$ for each odd $|\sigma_{k-1}|<j<|\sigma'_k|.$ Now, Player I plays $\sigma'_k(|\sigma'_k|-2)$ (considered a finite subset of $\omega)$ and when Player II responds with some $n_k\in \omega$, Player I defines $\sigma_k$ as $\sigma'_k {}^\frown \langle n_k,1\rangle$.
    Consider a match for the HMM game in which Player I played by following the previously defined strategy and let $\langle\sigma_k\rangle_{k\in\omega}$ be the sequence constructed along the game. 
    
    \begin{table}[H]
    		\centering
    		\begin{tabular}{l|llllll}
    			Player I  & $F_0$ &       & $F_1$ &       & $\dots$ &     \\  \hline
    			Player II &       & $n_0$ &       & $n_1$ &         & $\dots$
    		\end{tabular}
    	\end{table}
    From the construction, it should be clear that $\sigma=\bigcup\limits_{k\in\omega}\sigma_k$ is a branch through $T$, which means that $\{\sigma(j) : j\text{ is even and }\sigma(j+1)=1\}$ is an infinite member of $\mathcal{I}$. By the way in which  the strategy was defined, this element of the ideal must be equal to the set $\{n_k : k\in\omega\}$ which means that the match was won by Player II.

    (1) $\to$ (3) was proven in \cite{HMM2010}.

    (3) $\to$ (2) is just Corollary \ref{cor:winstrofII}.
    \end{proof}
    
    \begin{thm}\label{theoremtallness}For an ideal $\mathcal{I}$ in $\omega$, the following are equivalent:
    \begin{enumerate}
    \item Player II has a winning strategy for the HMM game with respect to $\mathcal{I}$.
    \item Player I has a winning strategy for the tallness game with respect to $\mathcal{I}$.
    \item There is an infinite branching tree $T \subseteq \omega^{<\omega}$ such that $f[\omega] \in \mathcal{I}^+$ for every $f \in [T]$.
    \item $\non^*(\mathcal{I}) = \aleph_0$.
    \end{enumerate}
    \end{thm}
    \begin{proof}
        The paper \cite{HMM2010} mentioned (1) $\leftrightarrow$ (3) $\rightarrow$ (4).
        
        We first show (3) $\to$ (2).
        Fix a tree $T$ witnessing (3).
        In the tallness game, first play the smallest child of the root. Until Player II plays $1$, play children of the root in ascending order.
        When Player II plays $1$, move to the children of the node Player I said previously. Repeat this process.
        This strategy is a winning strategy of Player I for the tallness game with respect $\idealI$.

        We next prove (2) $\to$ (1).
        Fix a winning strategy for the tallness game with respect to $\mathcal{I}$.
        Suppose Player I plays $F_0 \in [\omega]^{<\omega}$ for the HMM game.
        Until the play of Player I is greater than $\max F_0$, let Player II play $0$ in tallness game.
        When the play $n_{k_0}$ of Player I is greater than $\max F_0$, then let Player II play $1$ in tallness game and copy $n_{k_0}$ into the play of Player II in HMM game.
        Repeat this process.
        This strategy is a winning strategy of Player II for the HMM game with respect $\idealI$.
        This is because $\{ n_{k_i} : i \in \omega \} \in \idealI^+$, since Player I wins the tallness game.
        
        Finally, we prove (4) $\to$ (3).
        Let $\seq{X_n : n \in \omega}$ be a witness of $\non^*(\idealI) = \aleph_0$. That is, for every $I \in \idealI$, there is $n$ such that $\abs{X_n \cap I} < \aleph_0$.
        Consider the family $\seq{X_n \setminus m : n, m \in \omega}$ and rearrange this family into a sequence $\seq{Y_k : k \in \omega}$  of order type $\omega$.
        Then we have for every $I \in \idealI$, there is $k$ such that $Y_k \cap I = \varnothing$.
        Consider the uniform tree whose nodes in $k$-th level have successors $Y_k$.
        Then for every $f \in [T]$, we have $f[\omega] \in \idealI^+$.
    \end{proof}

    \begin{thm}For an ideal $\idealI$ in $\omega$, the following are equivalent:
    \begin{enumerate}
    \item Player I does not have winning strategy for the tallness* game with respect to $\idealI$.
    \item For any sequence $\langle X_n\,:n<\omega\rangle \subseteq \idealI^+$ there is $I\in \mathcal{I}$ such that $|I\cap X_n|=\omega$ for any $n\in\omega$.
    \end{enumerate}
    \begin{proof}
To see that $(1)$ implies $(2)$, let $\langle X_n\,:n<\omega\rangle \subseteq \idealI^+$ be a sequence of positive sets, and for each $n\in\omega$, increasingly enumerate $X_n$ as $\langle x_n^j\rangle_{j\in\omega}$. Finally, let $\phi:\omega\longrightarrow \omega$ be such that $\phi^{-1}[\{n\}]$ is infinite for each $n$. We define a strategy for Player I as follows: At the first turn, Player I defines $s_0=0$ and plays $\min X_{\phi(s_0)}$.  Now suppose that the $k$-th turn of the given match is finished and both players played according to the following table:
\begin{table}[H]
    		\centering
    		\begin{tabular}{l|llllll}
    			Player I  & $n_0$ &       & $\dots$ &       & $n_k $&     \\  \hline
    			Player II &       & $i_0$ &       & $\dots$ &         & $i_k$
    		\end{tabular}
    	\end{table}

If $i_k=0$, Player I defines $s_{k+1}$ as $s_k$. Otherwise, Player I let $s_{k+1}=s_k+1.$ Then Player I plays $n_{k+1}=\min\{m\in X_{\phi(s_{k+1})}\,:\,m>n_k\}$. This ends the definition of the strategy.\\
By the hypothesis, there is a match in which Player I played according to the previous strategy but Player I lost.

\begin{table}[H]
    		\centering
    		\begin{tabular}{l|llllll}
    			Player I  & $n_0$ &       & $n_1$ &       & $\dots$ &     \\  \hline
    			Player II &       & $i_0$ &       & $i_1$ &         & $\dots$
    		\end{tabular}
    	\end{table}
Observe that Player II must have played $1$ infinitely many times.
Otherwise, the sequence $\langle s_k \rangle_{k\in \omega}$ would be eventually constant, play with value $m$.
This would mean that $\{ n_k\,:\,k\in\omega\}=^* X_m$ which is a contradiction to the winning condition of Player II. 
We conclude $I=\{ n_k\,:\,i_k=1\}$ is an element of $\idealI$ and by definition of $\phi$ and the previous observation it is straightforward that $I\cap X_n$ is infinite for each $n$.

We now show (2) implies (1).
For this, suppose towards a contradiction that $\sigma\colon \omega^{<\omega}\to \omega$ is a winning strategy for Player I.
Let $M$ be a countable elementary submodel of a large enough $H(\lambda)$ with $\sigma,\idealI\in M$.
Applying the hypothesis (2) to the countable set $M \cap \idealI$, we can get $I \in \idealI$ such that $I\cap X$ is infinite for any $X\in M\cap \idealI^+$.
We define a play of Player II as follows: Suppose that we are at the $k$ turn and Player I has played $n_k$ in this turn. If $n_k\in I$, Player II plays $1$. Otherwise, Player II plays $0$.

Consider the match of the game in the previous paragraph.
\begin{table}[H]
    		\centering
    		\begin{tabular}{l|llllll}
    			Player I  & $n_0$ &       & $n_1$ &       & $\dots$ &     \\  \hline
    			Player II &       & $i_0$ &       & $i_1$ &         & $\dots$
    		\end{tabular}
    	\end{table}
As $\sigma$ is a winning strategy, it must be the case that $\{ n_k\,:\,k\in\omega\}\in \idealI^+$. 
Furthermore, since $\{ n_k : i_k = 1\}\subseteq I \in \idealI$, we have $\{n_k\,:i_k=1\}$ is finite. This means that $\{ n_k\,: i_k=1\}$ belongs to $M$.
Also, since $\sigma$ and $\idealI$ are in $M$, we have $\{n_k\,:k\in\omega\}$ also belongs to $M$. 
Hence $I\cap \{ n_k\,:\,k\in\omega\}$ is infinite but this set is equal to $\{n_k\,:i_k=1\}$. This contradiction finishes the proof.
\end{proof}
\end{thm}

\begin{lem}\label{lemmatallness*}For an ideal $\idealI$ in $\omega$, the following are equivalent:
\begin{enumerate}
\item Player II has a winning strategy for the tallness* game with respect to $\idealI$.
\item There is an $\idealI^*$-branching tree $T\subseteq \omega^{<\omega}$ such that $f[\omega]\in \idealI\cap [\omega]^{\omega}$ for every $f\in [T]$.
\end{enumerate} 
\begin{proof} First we prove that $(1)$ implies $(2)$. For this, take a winning strategy for Player II, play $\sigma:\omega^{<\omega}\longrightarrow 2$. Given $s\in \omega^{<\omega}$, let $$M_s=\{t\in \omega^{<\omega}\,:\, s\subsetneq t,\,\sigma(t)=1\text{ and }\sigma(t \upharpoonright j)=0\text{ for all }|s|\leq j<|t|\}$$
$$H_s=\{t(|t|-1)\,:\,t\in M_s\}.$$

We claim $H_s\in \idealI^*$.
To see this,  increasingly enumerate $\omega\backslash H_s$ as $\langle x_j \rangle _{\in \omega}$.
Let $j_0$ be such that $x_{j_0}>s(|s|-1)$ and consider the function $f= s^\frown \langle x_j\rangle_{j\geq j_0}$.
Observe that $f[\omega]=^* \omega\backslash H_s$.
Furthermore, as $\sigma$ is a winning strategy for Player II, we have either that $f[\omega]\in \idealI$ or $\{f(n):\sigma(f\upharpoonright(n+1))=1\}\in \idealI\cap[\omega]^{\omega}$.
It is easy to see that the second case would imply that $(\omega\backslash H_s)\cap H_s\not=\varnothing$, which is impossible.
Hence $f[\omega]\in \idealI$ and consequently $H_s\in \idealI^*$.

For each $s\in \omega^{<\omega}$ and $n\in H_s$ let $t^n_s\in M_s$ be such that $t^n_s(|t^n_s|-1)=n$.
Finally, define $T\subseteq \omega^{<\omega}$ together with a function $\phi:T\longrightarrow \omega^{<\omega}$ as follows: $\varnothing \in T$ and $\phi(\varnothing)=\varnothing$.
Having defined some $x\in T$ and $\phi(x)$, let $\suc_T(x)=H_{\phi(x)}$ and for each $n\in \suc_T(x)$, define $\phi(x^\frown n)$ as $t^n_{\phi(x)}$. It is straightforward that such $T$ works.

Now we prove that $(2)$ implies $(1)$.
For this, let $T\subseteq \omega^{<\omega}$ be as in the hypotheses. We define a strategy for Player II as follows: Before the game starts, Player II defines $s_0=\varnothing \in T$.
Suppose that we are in the $k$-th turn and Player II has already defined some $s_k\in T$.
If Player I plays some $n_k\in \suc_T(s_k)$ then Player II answers with $1$ and defines $s_{k+1}$ as $s_k \append n$.
Otherwise, Player II answers with $0$ and puts $s_{k+1}=s_k$.

We claim that the previous strategy is a winning one.
Indeed, consider a match of the game in which Player II played according to such strategy and let $\langle s_k\rangle_{k\in\omega}$ be the sequence in $T$ constructed along the way.
\begin{table}[H]
    		\centering
    		\begin{tabular}{l|llllll}
    			Player I  & $n_0$ &       & $n_1$ &       & $\dots$ &     \\  \hline
    			Player II &       & $i_0$ &       & $i_1$ &         & $\dots$
    		\end{tabular}
    	\end{table}

If $\{n_k\,:\, k\in\omega\}\in \idealI$ we are done, so suppose that this set belongs to $\idealI^+$. As the tree $T$ is $\idealI^*$-branching, it is easy to see that $f=\bigcup\limits_{k\in\omega}s_{k}\in[T]$ Furthermore, by definition of the strategy $f[\omega]$ is equal $\{n_k\,:\,i_k=1\}$. As $f[\omega]\in \mathcal{I}$ and is infinite, we are done.

\end{proof}
\end{lem}

\begin{lem}\label{IIstrategyed}Let $\idealI$ be an ideal over $\omega$. If $\idealI\geq_K \mathcal{ED}$ then Player II has a winning strategy for the tallness* game with respect to $\idealI$.
\begin{proof}
Let $\phi\colon\omega\to \omega\times \omega$ be such that $\phi^{-1}[I]\in \idealI$ for any $I\in \mathcal{ED}$.
The following is an $\mathcal{ED}^*$-branching tree: $$T=\{\langle (n_0,m_0),\dots, (n_{k-1},m_{k-1})\rangle \in (\omega\times \omega)^{<\omega} \, :\, k\in\omega\text{ and }n_i<n_j\text{ for all }i<j<k\, \}$$
This is because the set of successors of any element of $T$ is of the form $(\omega\backslash n)\times \omega $ for some $n\in\omega$.
Note that for any $f\in [T]$ it holds that $f[\omega]$ is the graph of a partial function which means that $f[\omega]\in \mathcal{ED}$.
To finish, just define a tree $T_{\phi}\subseteq \omega^{\omega}$ as follows:
$$\langle n_0,\dots,n_{k-1}\rangle \in T\text{ if and only if }\langle \phi(n_0),\dots, \phi(n_{k-1})\rangle\in T.$$

It should be clear that $T_{\phi}$ is an $\idealI^*$ branching tree because $\phi^{-1}[U]\in \idealI^*$ for any $U\in \mathcal{ED}^*$.
Furthermore,  $\idealI$ for any $f\in [T_\phi]$ we have that $\phi\circ f\in [T]$.
This means that $(\phi\circ f)[ \omega]\in \mathcal{ED} $.
Since  $f[\omega]\subseteq \phi^{-1}[\phi\circ f[\omega]]$ it follows that $f[\omega]\in \idealI$.
By Lemma \ref{lemmatallness*}, the proof is over.

\end{proof}
\end{lem}
 
As a corollary of Lemmas \ref{lemmatallness*}, \ref{IIstrategyed} in this article, and the discussion after Claim 3.3 in \cite{katetovorderhrusak} we have the following theorem.
\begin{thm}Let $\idealI$ be an ideal over $\omega$. The following are equivalent:
\begin{enumerate}
\item There is $X\in \idealI^+$ for which Player II has a winning strategy in the tallness* game with respect to $\idealI|_X$.
\item There is an $X\in \idealI^+$ for which $\idealI|_X\geq_K\mathcal{ED}$.
\end{enumerate}
\end{thm}

\section{Two versions for games on ideals}\label{sec:twoversions}

In this section, we generalize games in Section \ref{sec:tallness} by means of Definition \ref{defBgame}. At the same time, in Definition \ref{defGgame} we introduce the game $\mathfrak{G}(\idealI,\idealJ)$. Particular instances of this game correspond to the filter games introduced by C. Laflamme in \cite{Laflamme}. As communicated by M. Hru\v{s}\'ak, the Game $\mathfrak{G}(\idealI,\idealJ)$ was used by him and A. Shibakov in \cite{hrusakzindulka} in order to prove what they call \textit{the Definable ideal dichotomy}. The main results of this section are Theorems \ref{winningIBG} and \ref{winningIIBG}. In there, we show that winning strategies for one player in the game $\mathfrak{G}(\idealI,\idealJ)$ correspond to winning strategies for the other player in the game $\mathfrak{B}(\idealI,\idealJ)$. As a corollary of both theorems, we also prove the Definable ideal dichotomy.

 \begin{defi}\label{defGgame}Let $\idealI$ and $\idealJ$ be two ideals over $\omega$. We define the game $\mathfrak{G}(\idealI,\idealJ)$ as follows:
\begin{table}[H]
		\centering
		\begin{tabular}{l|llllll}
			Player I  & $I_0\in \idealI$ &       & $ I_1\in\idealI$ &       & $\dots$ &     \\  \hline
			Player II &       & $n_0\not\in I_0$ &       & $n_0<n_1\not\in I_1$ &         & $\dots$
		\end{tabular}
	\end{table}

  Here, Player II wins whenever $\langle n_i\rangle_{i\in\omega}\in \idealJ^+.$ 
 \end{defi}
Note that $\mathfrak{G}(\text{Fin},\idealJ)$ is just the HMM game for the ideal $\idealJ$ and that $\mathfrak{G}(\idealI,\idealI)$ is just one instance of the games defined by Laflamme in \cite{Laflamme}.
 \begin{defi}\label{defBgame}
 Let $\idealI$ and $\idealJ$ be two ideals over $\omega$. We define the game $\mathfrak{B}(\idealI,\idealJ)$ as follows:
\begin{table}[H]
		\centering
		\begin{tabular}{l|llllll}
			Player I  & $n_0\in\omega $ &       & $n_0<n_1\in\omega$ &       & $\dots$ &     \\  \hline
			Player II &       & $i_0\in 2$ &       & $i_1\in 2$ &         & $\dots$
		\end{tabular}
	\end{table}
Here, Player II wins whenever $\langle n_k\rangle_{k\in\omega}\in \idealI$ or $\{ n_k\,:\,i_k=1\}\in \idealJ\cap [\omega]^\omega$.
 \end{defi}
Note that $\mathfrak{B}(\mathsf{Fin},\idealJ)$ is the tallness game with respect to $\idealJ$ and $\mathfrak{B}(\idealI,\idealI)$ is the tallness* game with respect to $\idealI$
The following theorem is proved in the same way as  Lemma \ref{lemmatallness*}. 

\begin{thm}\label{winningIBG}Let $\idealI$ and $\idealJ$ be two ideals over $\omega$. The following are equivalent:
\begin{enumerate}
\item Player I has a winning strategy in $\mathfrak{G}(\idealI,\idealJ)$.
\item Player II has a winning strategy in $\mathfrak{B}(\idealI,\idealJ).$

\item There is an $\idealI^*$-branching tree $T\subseteq \omega^{<\omega}$ such that $f[\omega]\in \idealJ\cap [\omega]^{\omega}$ for every $f\in [T]$.
\end{enumerate}
\begin{proof}
First we prove that (1) implies (2). For this, let $\sigma: \omega^{<\omega}\longrightarrow \idealI$ be a winning strategy for Player I in $\mathfrak{G}(\idealI,\idealJ)$. We define a strategy for Player II in $\mathfrak{B}(\idealI,\idealJ)$ together with a sequence $\langle t_k\rangle_{k\in\omega}$ of elements 
of $\omega^{<\omega}$ as follows:  At the first turn Player I plays some $n_0\in\omega$. If $n_0\not\in \sigma(\varnothing)$ then Player II defines $t_0$ as $\langle n_0 \rangle$ and plays $i_0=1$. Otherwise, Player II defines $t_0$ as the empty sequence and plays $i_0=0$. In 
general, if we are on the $k+1$ turn and Player I plays some $n_{k+1}\in\omega$ then Player II defines $t_{k+1}$ as $t_k^\frown \langle n_{k+1}\rangle$ if $n_{k+1}\not\in \sigma(t_k)$ and plays $i_{k+1}=1$. Otherwise, Player II defines $t_{k+1}$ as $t_k$ and plays $i_{k+1}=0$. 

The previously defined strategy is a winning one. For if $\langle t_k\rangle_{k\in\omega}$ is an eventually constant sequence, play with value $t$, then $\{ n_k\,:\,k\in\omega\}\subseteq^*\sigma(t)\in \idealI$. Otherwise, $$f=\bigcup\limits_{k\in\omega}t_k\in \omega^\omega$$ represents a match of the game $\mathfrak{G}(\idealI,\idealJ)$ in which Player I played according to $\sigma$. As this is a winning strategy for Player I, it should happen that $\{ n_k\,:\,i_k=1\}=f[\omega]\in \mathcal{J}\cap[\omega]^\omega$.

Now we prove that (2) implies (3). For this, take a winning strategy $\sigma:\omega^{<\omega}\longrightarrow 2$ for Player II in $\mathfrak{B}(\idealI,\idealJ)$. Given $s\in \omega^{<\omega}$, let $$M_s=\{t\in \omega^{<\omega}\,:\, s\subsetneq t,\,\sigma(t) = 1\text{ and }\sigma(t \upharpoonright j) = 0 \text{ for all }|s|\leq j<|t|\}$$
$$H_s=\{t(|t|-1)\,:\,t\in M_s\}.$$
We claim $H_s\in \idealI^*$. To see this,  increasingly enumerate $\omega\backslash H_s$ as $\langle x_j \rangle _{\in \omega}$. Let $j_0$ be such that $x_{j_0}>s(|s|-1)$ and consider the function $f= s^\frown \langle x_j\rangle_{j\geq j_0}$. Observe that $f[\omega]=^* \omega\backslash H_s$. Furthermore, as $\sigma$ is a winning strategy for Player II, we have either that $f[\omega]\in \idealI$ or $\{f(n)\,:\sigma(f \upharpoonright (n+1))=1\}\in \idealJ\cap[\omega]^{\omega}$. It is easy to see that the second case would imply that $(\omega\backslash H_s)\cap H_s\not=\varnothing$, which is impossible. Hence $f[\omega]\in \idealI$ and consequently $H_s\in \idealI^*$.

For each $s\in \omega^{<\omega}$ and $n\in H_s$ let $t^n_s\in  M_s$ be such that $t^n_s(|t^n_s|-1)=n$. Finally, define $T\subseteq \omega^{<\omega}$ together with a function $\phi:T\longrightarrow \omega^{<\omega}$ as follows: $\varnothing \in T$ and $\phi(\varnothing)=\varnothing$. Having defined some $x\in T$ and $\phi(x)$, let $\suc_T(x)=H_{\phi(x)}$ and for each $n\in \suc_T(x)$, define $\phi(x^\frown n)$ as $t^n_{\phi(x)}$. It is straightforward that such $T$ works.

Finally, we prove that (3) implies (1). For this, let $T$ be a tree as in the hypotheses. We define a strategy for Player I in $\mathfrak{G}(\mathcal{I},\mathcal{J})$ as follows: On the first turn Player I plays $I_0= \omega\backslash \suc_T(\varnothing)$. Now suppose that we are on the $k+1$ turn and that both players have played according to the following table:

\begin{table}[H]
		\centering
		\begin{tabular}{l|llllll}
			Player I  & $I_0\in \idealI$ &       & $\dots$ &       & $I_k\in \idealI$ &     \\  \hline
			Player II &       & $n_0\not\in I_0$ &       & $\dots$ &         & $n_k\notin I_k$
		\end{tabular}
	\end{table}

Furthermore, suppose that $\langle n_0,\dots n_k\rangle\in T$. Then Player I plays \linebreak
$I_{k+1}=\omega\setminus \suc_T(\langle n_0,\dots,n_k\rangle)$. Observe that Player II is forced to play some $n_{k+1}\in \suc_T(\langle n_0,\dots,n_k\rangle)$. In other words, $\langle n_0,\dots,n_k,n_{k+1}\rangle\in T$. This shows that Player I can play the next turn following the strategy defined above.

By definition, for any match played according to the previously defined strategy, it holds that $\langle n_k\rangle_{k\in\omega}\in [T]$. That is, $\{n_k\,:\,k\in\omega\}\in \mathcal{J}$. This finishes the proof.
\end{proof}
\end{thm}

The proof of the next theorem is analogous to one of Theorem \ref{theoremtallness}. We already showed how to translate the proof of Lemma \ref{lemmatallness*} in its full generality in Theorem \ref{winningIBG} and the translation of Theorem \ref{theoremtallness} to Theorem \ref{winningIIBG} is carried in a similar manner.

\begin{thm}\label{winningIIBG}If $\idealI$ and $\idealJ$ are two Borel ideals over $\omega$, the following are equivalent:
\begin{enumerate}
\item Player II has a winning strategy in $\mathfrak{G}(\idealI,\idealJ)$.
\item Player I has a winning strategy in $\mathfrak{B}(\idealI,\idealJ).$
\item There is an $\idealI^+$-branching tree $T \subseteq \omega^{<\omega}$ such that $f[\omega] \in \mathcal{J}^+$ for every $f \in [T]$.
\end{enumerate}

\end{thm}

 As a corollary, we have the following:
 \begin{thm}[Definable Ideal Dichotomy]Given a filter $\mathcal{F}$ and a Borel ideal $\mathcal{I}$, the following are equivalent:
 \begin{enumerate}
    \item There is an $\mathcal{F}^+$-branching tree $T \subseteq \omega^{<\omega}$ such that $f[\omega] \in \mathcal{I}^+$ for every $f \in [T]$.
    
\item There is an $\mathcal{F}^*$-branching tree $T\subseteq \omega^{<\omega}$ such that $f[\omega]\in \idealI\cap [\omega]^{\omega}$ for every $f\in [T]$.

 \end{enumerate}
 \begin{proof}Even though the filter $\mathcal{F}$ may not be definable the game $\mathfrak{G}(\mathcal{F}^*,\mathcal{I})$ is a determined game. To see this, note that the tree $G$ associated to the game can be viewed as the subtree of $A^{<\omega}$ where $A=\mathcal{F}\cup \omega$ and $\sigma\in G $ if and only if for any $n\in \dom(\sigma)$ it happens that $\sigma(n)\in \mathcal{F}$ whenever $n$ is even and $\sigma(n-2)<\sigma(n)\in\omega\backslash \sigma(n-1)$  whenever $n$ is odd. Consider the  function $\pi:[G]\longrightarrow [\omega]^{\omega}$ given by $\pi(f)=f[\{2n+1\,:\,n\in \omega\}]$. It is direct that $\pi$ is continuous and that the payoff set for Player $I$ is equal to $\pi^{-1}[\mathcal{I}^+]$. Since $\mathcal{I}$ is Borel,  we are done.
 \end{proof}
\end{thm}

\section{Acknowledgements}

    This paper is developed for the proceedings of the RIMS Set Theory Workshop 2024 \textit{Recent Developments in Axiomatic Set Theory} held at RIMS,  Kyoto University. The authors thank the organizer,   Masahiro Shioya and the co-organizer, Kenta Tsukuura  for allowing them to give a talk at the workshop and to submit a paper to this proceedings.

    The authors are grateful to J\"{o}rg Brendle, Osvaldo G\'uzman and Michael Hru\v{s}\'ak for their  helpful comments.
    This work was supported by JSPS KAKENHI Grant Number JP22J20021.

	\printbibliography
\end{document}